\documentclass[12pt]{amsart}%
\usepackage{amsmath,amsthm,amsfonts,amssymb}
\usepackage{verbatim}
\usepackage{latexsym}
\usepackage{marginnote}
\usepackage{amscd}
\usepackage{amssymb}
\usepackage{xcolor}
\usepackage{amsmath}
\usepackage{amsfonts}
\usepackage{mathrsfs}
\usepackage{graphicx}%
\providecommand{\U}[1]{\protect\rule{.1in}{.1in}}

\providecommand{\U}[1]{\protect\rule{.1in}{.1in}}
\providecommand{\U}[1]{\protect\rule{.1in}{.1in}}
\newtheorem{theorem}{Theorem}[section]

\newtheorem{corollary}[theorem]{Corollary}

\newtheorem{remark}[theorem]{Remark}

\theoremstyle{definition}

\begin{document}
\title{Spherical Fourier multipliers related to Gelfand pairs}
\author{ Yaogan Mensah and Marie Fran\c coise Ouedraogo}
\address{Department of Mathematics, University of Lom\'e, 01 PO Box 1515, Lom\'e 1, Togo}
\email{\textcolor[rgb]{0.00,0.00,0.84}{mensahyaogan2@gmail.com, ymensah@univ-lome.org}}
\address{Department of Mathematics, University Joseph Ki-Zerbo, 03 PO Box 7021, Ouagadougou 03, Burkina Faso}
\email{\textcolor[rgb]{0.00,0.00,0.84}{omfrancoise@yahoo.fr, francoise.ouedraogo@ujkz.bf}}

\begin{abstract}
In this paper,  we introduce a family of Fourier multipliers using the spherical Fourier transform on Gelfand pairs. We refer to them  as spherical Fourier multipliers. We study certain sufficient conditions under which they are bounded. Then, under the hypothesis of compactness of the underlying group and under certain summability conditions, we obtain the belonging of the spherical Fourier multipliers to some Schatten-von Neumann classes.
\end{abstract}
\maketitle


\section{Introduction}
Like the Fourier transformation and other transformations related to it, Fourier multipliers are present and play a fundamental role in classical harmonic analysis \cite{Grafakos} or abstract harmonic analysis \cite{Larsen}. They are a  class of  operators defined from the Fourier transform and they are mathematical tools that appear for instance in signal analysis and Partial Differential equations. Fourier multipliers are closely linked to localization operators \cite{Wong}.

The most important aspect that scholars study on these operators is their boundedness. Several authors have studied this problem in various situations. We can cite among others \cite{Anker, Hormander,Ruzhansky, Soltani}.  
Recent work related to Fourier multipliers  are \cite{Blasco, Catana,Kumar} to name just a few. They can be defined from the Fourier transform related to classical groups such as $\mathbb{R}^n$, the $n$-torus $\mathbb{T}^n$, or from more abstract groups such as locally compact abelian groups or locally compact nonabelian groups.  

 The theory of Gelfand pair first appeared in the Gelfand paper \cite{Gelfand}. This theory makes it possible to construct a Fourier transformation (called a spherical Fourier transformation) from a locally compact group and one of its compact subgroups chosen in such a way that the space of integrable functions on the group is commutative under the convolution product.

 In this paper, we study Fourier multipliers related to the  spherical Fourier transform on Gelfand pairs. In this context, we call them {\it spherical Fourier multipliers}. As results, we obtain sufficient conditions of continuity of the spherical Fourier multipliers. Moreover, under the condition of compactness, we obtain results related to boundedness and  belonging to Schatten-von Neumann spaces similar to localization operators in time-frequency analysis.

The rest of the paper is organized as follows. In Section \ref{Preliminaries}, we summarize the mathematical tools necessary for an independent understanding of the article. We essentially recall notions linked to the convolution product, spherical harmonic analysis, Schatten-von Neumann spaces and interpolation between  spaces of functions or operators. In Section \ref{Fourier multiplier operators}, we study various properties of spherical Fourier multipliers on Gelfand pairs, particularly in terms of boundedness. In Section \ref{The spherical Fourier multiplier operators and the Schatten-von Neumann classes}, under the assumption that the group is compact, we proved sufficient conditions under which spherical Fourier multipliers belong to Schatten-von Neumann classes.

\section{Preliminaries}\label{Preliminaries}

\subsection{Convolution}\label{convolution}
  Our references for this subsection are \cite{Folland, Hewitt}.
  Let $G$ be a locally compact and Hausdorff group with neutral element $e$ and with  a fixed left Haar measure.  We denote by $\mathcal{C}_c(G)$ the space  of complex valued continuous functions on $G$ with compact support.  There exists a homomorphism $\Delta : G\longrightarrow (0,\infty)$ such that for each $f\in \mathcal{C}_c(G)$, 
$$\displaystyle\int_G f(xy^{-1})dx=\Delta (y)\int_G f(x)dx.$$
If for all $y\in G,\,\Delta(y)=1$, then the group $G$ is said to be unimodular. 

 The Lebesgue spaces on $G$ are denoted by  $L^p(G)$,  $1\le p \le \infty$. These spaces are endowed  with the following norms under which they are Banach spaces : 
$$\|f\|_{L^p(G)}=\left( \int_G |f(x)|^pdx\right)^{\frac{1}{p}}, \, 1\le p <\infty$$
and 
$$\|f\|_{L^\infty(G)}=\mbox{sup ess}|f|.$$
  
 If $f,g\in L^1(G)$,  the convolution product of $f$ and $g$ is defined by
$$(f\ast g) (x)=\int_G f(y)g(y^{-1}x)dy.$$
The convolution product is commutative, that is $f\ast g= g\ast f,\,\forall f,g\in L^1(G)$,  if and only if, the group $G$ is   commutative. A Banach algebra structure is constructed on $L^1(G)$ with respect to the convolution product. More precisely, we have
$$\forall f,g\in L^1(G),\,\|f\ast g\|_{L^1(G)}\le\|f\|_{L^1(G)}\|g\|_{L^1(G)}.$$
The convolution product is extended to functions in $L^p(G)$. The following facts about the convolution  product in link with the $L^p$-spaces will be useful in the sequel : 
\begin{itemize}
\item Let $1\le p\le \infty$. If $f\in L^1(G)$ and $g\in L^p(G)$, then $f\ast g\in L^p(G)$ and $$\|f\ast g\|_{L^p(G)}\le \|f\|_{L^1(G)}\|g\|_{L^p(G)}.$$
\item Suppose $G$ is unimodular.  Let $1<p,q<\infty$ with $\displaystyle\frac{1}{p}+\displaystyle\frac{1}{q}=1$. If $f\in L^p(G)$ and $g\in L^q (G)$, then $f\ast g\in \mathcal{C}_0 (G)$ and $$\|f\ast g\|_\infty\le \|f\|_{L^p(G)}\|g\|_{L^q(G)},$$ where $\mathcal{C}_0 (G)$ is the space of complex  continuous functions that tend to zero at infinity. 
\item Suppose $G$ is unimodular. Let $1<p,q,r<\infty$ with $\displaystyle\frac{1}{p}+\displaystyle\frac{1}{q}=\frac{1}{r}+1$. If $f\in L^p(G)$ and $g\in L^q (G)$, then $f\ast g\in  L^r(G)$ and $$\|f\ast g\|_{L^r(G)}\le \|f\|_{L^p(G)}\|g\|_{L^q(G)}.$$.
\end{itemize}
 
\subsection{Harmonic analysis on Gelfand pairs}
Our  references for this subsection are \cite{Dijk,Wolf}.  
Let $G$ be a locally compact and Hausdorff group. Let $K$ be a compact subgroup of $G$. 

A function $f: G\longrightarrow \mathbb{C}$ is said to be $K$-bi-invariant if  $\forall\, k_1,k_2\in K,\,\forall x\in G$,
$$f(k_1xk_2)=f(x)$$
 We denote by $\mathcal{C}_c^\natural(G)$ the space of complex valued continuous $K$-bi-invariant functions on $G$ with compact support and by $L^{1,\natural}(G)$  
 the space of complex valued integrable functions which are $K$-bi-invariant on $G$. More generally,  $L^{p,\natural}(G), 1\le p\le \infty$,  will denote
 the space of complex valued $p$-integrable functions which are $K$-bi-invariant on $G$.
 
  By definition,  $(G,K)$ is called a Gelfand pair if $L^{1,\natural}(G)$ is a commutative  convolution algebra \cite[page 153]{Wolf}. Since $\mathcal{C}_c^\natural(G)$ is dense in $L^{1,\natural}(G)$, then it is equivalent to say that $(G,K)$ is  a Gelfand pair if and only if $\mathcal{C}_c^\natural(G)$  is commutative under the convolution product. 
 It is well-known that  if $(G,K)$ is a Gelfand pair, then $G$  is unimodular \cite[page 154]{Wolf}.
 
 Let $(G,K)$ be a Gelfand pair. A function $\chi :  \mathcal{C}_c^\natural(G)\longrightarrow\mathbb{C}$ is called a character if 
 $$\forall  f,g\in \mathcal{C}_c^\natural(G),\,\chi (f\ast g)= \chi (f)\chi (g).$$
 Let $\varphi$ be a $K$-bi-invariant function on $G$. Then, one says that $\varphi$ is a  spherical function if   the function $\chi_\varphi : \mathcal{C}_c^\natural(G)\longrightarrow\mathbb{C}, f\longmapsto \displaystyle\int_G f(x)\varphi (x)dx$ is a nontrivial character. 
 
  The following theorem gives characterizations of spherical functions.
 \begin{theorem}(\cite[page 77]{Dijk}, \cite[page 157]{Wolf})\label{properties spherical functions}
 
 The following assertions are equivalent.
 \begin{enumerate}
 \item The function $\varphi$ is a spherical function. 
 \item  The function $\varphi$ is $K$-bi-invariant with $\varphi (e)=1$  and such that for each $f\in \mathcal{C}_c^\natural(G)$, there exists a complex number $\lambda_f$ such that $f\ast \varphi=\lambda_f\varphi$.
 \item  The function $\varphi$ is continuous, not identically zero and if $x,y\in G$, then $$\displaystyle\int_K\varphi (xky)dk=\varphi (x)\varphi (y).$$ 
 \end{enumerate}
 \end{theorem}
 
We denote by $\mathcal{S}^b(G,K)$ the set of bounded spherical functions for the  Gelfand pair $(G,K)$. However, by simplicity, we will write $\mathcal{S}^b$ for $\mathcal{S}^b(G,K)$.  
  For a function $f\in L^{1,\natural}(G)$, the spherical Fourier transform of $f$, denoted by $\mathcal{F}f$ or $\widehat{f}$, is defined by $$\mathcal{F}{f}(\varphi)=\displaystyle\int_Gf(x)\varphi (x^{-1})dx, \, \varphi\in \mathcal{S}^b.$$
The set  $\mathcal{S}^b$ inherits a weak topology from the family $\lbrace \widehat{f} : f\in L^{1,\natural}(G)\rbrace$ and  $\mathcal{S}^b$ is a locally compact and Hausdorff space under this topology \cite[page 185]{Wolf}.   
  
  In order to obtain the inversion formula, we recall the notion of positive definite function.
  
A function $\varphi: G\longrightarrow \mathbb{C}$ is said to be positive  definite if $\forall N\in \mathbb{N},\forall x_1,\cdots, x_N\in G, \forall z_1,\cdots,z_N\in \mathbb{C}$ the following inequality holds : 
$$\sum_{n=1}^N\sum_{m=1}^N\varphi (x_n^{-1}x_m)\overline{z_n}z_m\ge 0.$$
 The positive functions have the following properties.
\begin{theorem}\cite[page 165]{Wolf}. \label{Properties psd functions}

If $\varphi: G\longrightarrow \mathbb{C}$ is a positive definite function, then
\begin{enumerate}
\item $\forall x\in G,\, |\varphi(x)|\le \varphi(e)$.
\item $\forall x\in G,\,\varphi(x^{-1})=\overline{\varphi(x)}$.
\end{enumerate}
\end{theorem} 
We denote by $\mathcal{S}^+$ the set of positive definite spherical functions for the Gelfand pair $(G,K)$. 
From Theorem \ref{properties spherical functions}(2) and Theorem \ref{Properties psd functions}(1), it is clear that $\mathcal{S}^+$ is uniformly bounded in the sense that 
$$\forall \varphi\in \mathcal{S}^+,\forall x\in G, |\varphi(x)|\le \varphi(e)=1.$$
 
Let us denote by $B^\natural (G)$ the set of  linear combinations of positive definite and $K$-bi-invariant functions on $G$. In \cite[page 191]{Wolf}, it was proved that there is a positive Radon measure $\mu$ on $\mathcal{S}^+$ such that if $f\in B^\natural (G)\cap L^{1,\natural}(G)$,  then $\widehat{f}\in L^1(\mathcal{S}^+)$ and the following spherical  Fourier  inversion formula holds :
 $$\forall x\in G,\,f(x)=\displaystyle\int_{\mathcal{S}^+}\widehat{f}(\varphi)\varphi (x)d\mu (\varphi).$$

Let us recall  some properties of the spherical Fourier transform.

Let us denote by $\mathcal{C}_{0}(\mathcal{S}^b)$ the set of complex valued functions on $\mathcal{S}^b$ which vanish at $\infty$.
\begin{theorem}\cite[page 185]{Wolf}\label{Riemann-Lebesgue}

If $f\in L^{1,\natural}(G)$, then $\widehat{f}\in \mathcal{C}_{\infty}(\mathcal{S}^b)$ and 
$\|\widehat{f}\|_\infty\le \|f\|_{L^{1,\natural}(G)}.$
\end{theorem}

\begin{theorem}\cite[page 193]{Wolf}(Plancherel formula)\label{Plancherel}

Let $(G,K)$ be a Gelfand pair. If $f\in L^{1,\natural}(G)\cap L^{2,\natural}(G)$, then $\widehat{f}\in L^2(\mathcal{S}^+)$ and $\|\widehat{f}\|_{ L^2(\mathcal{S}^+)}=\|f\|_{L^{2,\natural}(G)}$. Moreover, the spherical Fourier transform $\mathcal{F}:L^{1,\natural}(G)\cap L^{2,\natural}(G)\longrightarrow L^2(\mathcal{S}^+)$ extends by $L^2$-continuity to an isometry from $L^{2,\natural}(G)$ onto $L^2(\mathcal{S}^+)$.
\end{theorem}
\begin{corollary}\cite[page 194]{Wolf}\label{Plancherelscalar}

If $f,g\in  L^{2,\natural}(G)$, then $\widehat{f},\widehat{g}\in L^2(\mathcal{S}^+)$ and 
$$\langle \widehat{f}, \widehat{g}\rangle_{L^{2}(\mathcal{S}^+)}=\langle f,g\rangle_{L^{2,\natural}(G)}.$$
\end{corollary}

\subsection{Schatten-von Neumann classes}
Our references for this part are  \cite{ Dunford,Wong}.

Let $H$ be a separable complex Hilbert space. Let $T$ be a bounded operator on $H$. Let us denote by $T^\ast$ its adjoint operator.  We define the operator $|T|: H \longrightarrow H$ by $|T|=\sqrt{T^\ast T}$. Obviously, $|T|$ is a positive operator. Now, let $T$ be a compact operator. We denote by $s_k(T), \,k=1, 2,\cdots$, the eigenvalues of $|T|$ (such eigenvalues are called singular values of $T$). 
A compact operator $T:H\longrightarrow H$ is said to be in the Schatten-von Neumann class $S_p(H), 1\le p<\infty,$ if 
$$\sum_{k=1}^\infty (s_k(T))^p<\infty.$$
The space $S_p(H)$ is a complex Banach space when it is endowed with the norm :
$$\|T\|_{S_p(H)}=\left(\sum_{k=1}^\infty (s_k(T))^p\right)^{\frac{1}{p}}.$$
$S_1(H)$ and $S_2(H)$ are customaryly called the trace class and  the Hilbert-Schmidt class respectively. By convention, $S_\infty(H)=\mathcal{B}(H)$, the space of bounded operators on $H$.  

The following  result may be useful.

\begin{theorem}\cite[page 15, Proposition 2.4]{Wong}

Let $T: H\longrightarrow H$ be a positive operator. If $\sum_{k=1}^\infty \langle Te_k,e_k\rangle<\infty$ for all orthonormal bases $\{e_k : k=1,2,\cdots\}$ of $H$, then $T$ is in the trace class $S_1(H)$.

\end{theorem}

\subsection{Interpolation theorems}
Our references for this part are \cite{Tartar, Wong, Zhu}.
Let us briefly recap the results we will need when it comes to interpolating between function spaces. The follwing theorem  is known as the Riesz-Thorin interpolation theorem orRiesz-Thorin convexity theorem. 
\begin{theorem}\cite[page 104]{Tartar}\label{Riesz-Thorin}

Let $1 \le p_0, p_1, q_0, q_1 \le\infty$. If $T : L^{p_0} (\mu) \to L^{q_0}(\nu)$ is a bounded linear operator  with norm $M_0$ and $T : L^{p_1} (\mu) \to L^{q_1}(\nu)$ is a bounded linear operator  with norm $M_1$, then  for $0 < \theta< 1$, $T : L^{p} (\mu) \to L^{q}(\nu)$ is a bounded linear operator with norm $M \le M_0^{ 1-\theta}M_1^{\theta}$, where
$\displaystyle \frac{1}{p} = \frac{1-\theta}{p_0} + \frac{\theta}{p_1}$ and $\displaystyle\frac{1}{q} = \frac{1-\theta}{q_0} + \frac{\theta}{q_1}$. 
\end{theorem}
For the Lebesgue spaces $L^p(X,\mu)$  and the
Schatten-von Neumann class $S_p(H)$ with  $1\le  p \le \infty$, we have the following interpolation results. 
\begin{theorem}\cite[page 20]{Wong}\label{InterpolationParticular}

For $1\le  p \le \infty$,
\begin{enumerate}
\item $[L^1(X,\mu), L^\infty(X,\mu)]_{\frac{1}{p'}}=L^p(X,\mu),$
\item  $[S_1(H), S_\infty(H)]_{\frac{1}{p'}}=S_p(H),$
\end{enumerate}
where $p'$ is such that $\displaystyle\frac{1}{p}+\frac{1}{p'}=1$.
\end{theorem}

\section{Spherical Fourier multipliers}\label{Fourier multiplier operators}
In this section, we define a spherical Fourier multiplier  related to the Gelfand pair $(G,K)$ by the means of the spherical Fourier transform. 

For a function  $m :\mathcal{S}^+\longrightarrow\mathbb{C}$, we define the spherical Fourier multiplier  by the formal expression
\begin{equation}
T_mf (x)=\displaystyle\int_{\mathcal{S}^+}m(\varphi)\widehat{f}(\varphi)\varphi (x)d\mu (\varphi),\, x\in G.
\end{equation}
One can observe that 
\begin{equation}\label{equation:FourierDeLm}
\mathcal{F}(T_mf)=m\widehat{f}.
\end{equation}

\begin{theorem} 
Let $f,g\in L^{1,\natural}(G))$. Then, the following equalities hold. 
\begin{enumerate}
\item $T_m(f\ast g)=(T_mf)\ast g,$
\item $T_{m_1}f\ast  T_{m_2}g=T_{m_1m_2}(f\ast g).$
\end{enumerate}
\end{theorem}

\begin{proof}

\begin{itemize}
\item 
Using Formula (\ref{equation:FourierDeLm}), we have
$$\mathcal{F}(T_m(f\ast g))=m\widehat{f\ast g}=m\widehat{f}\widehat{g}=\mathcal{F}(T_mf)\widehat{g}.$$
Thus, $T_m(f\ast g)=(T_mf)\ast g.$\\
\item Again by Formula (\ref{equation:FourierDeLm}), we have
\begin{align*}
\mathcal{F}(T_{m_1m_2}(f\ast g))&=m_1m_2\widehat{f\ast g}\\
&=m_1m_2\widehat{f}\widehat{g}\\
&=m_1\widehat{f}m_2\widehat{g}\\
&=\mathcal{F}(T_{m_1}f)\mathcal{F}(L_{m_2}g)\\
&=\mathcal{F}(T_{m_1}f\ast TL_{m_2}g)
\end{align*}
This implies that $T_{m_1}f\ast  T_{m_2}g=T_{m_1m_2}(f\ast g).$
\end{itemize}
\end{proof}

The following result identifies the Hilbert adjoint of the operator $T_m$. 

\begin{theorem} 
The adjoint of the operator $T_m : L^{2,\natural}(G)\longrightarrow L^{2,\natural}(G)$ is the operator $T_m^\ast : L^{2,\natural}(G)\longrightarrow L^{2,\natural}(G)$  defined by $T_m^\ast f=T_{\overline{m}}f$, where $\overline{m}$ is the complex conjugate  of $m$.
\end{theorem}

\begin{proof} Let $f,g\in L^{2,\natural}(G)$. Then, 
\begin{align*}
\langle T_mf,g\rangle_{L^{2,\natural}(G)} &= \langle \mathcal{F}(T_mf),  \widehat{g}\rangle_{L^{2}(\mathcal{S}^+)}\,\mbox{(Corollary \ref{Plancherelscalar})}\\
&= \langle m\widehat{f},  \widehat{g}\rangle_{L^{2}(\mathcal{S}^+)}\\
&=\displaystyle\int_{\mathcal{S}^+}m(\varphi)\widehat{f}(\varphi)\overline{\widehat{g}(\varphi)}d\mu (\varphi)\\
&=\displaystyle\int_{\mathcal{S}^+}\widehat{f}(\varphi)\overline{\overline{m(\varphi)}\widehat{g}(\varphi)}d\mu (\varphi)\\
&= \langle \widehat{f},  \overline{m}\widehat{g}\rangle_{L^{2}(\mathcal{S}^+)}\\
&= \langle \widehat{f},  \mathcal{F}(T_{\overline{m}}g)\rangle_{L^{2}(\mathcal{S}^+)}\\
&= \langle f, T_{\overline{m}}g\rangle_{L^{2,\natural}(G)}.
\end{align*}
Thus, $T_m^\ast=T_{\overline{m}}.$
\end{proof}

\begin{theorem}\label{theorem:boundedL1} 
If  $m\in L^{1}(\mathcal{S}^+)$,  then $T_m : L^{1,\natural}(G)\longrightarrow L^{\infty,\natural}(G)$ is bounded and  $$\|T_m\|\le \|m\|_{L^{1}(\mathcal{S}^+)}$$
\end{theorem}

\begin{proof}
We have 
\begin{align*}
|T_mf(x)|&\le \displaystyle\int_{\mathcal{S}^+}|m(\varphi)|| \widehat{f}(\varphi)||\varphi (x)|d\mu (\varphi)\\ 
 &\le  \displaystyle\int_{\mathcal{S}^+}|m(\varphi)|\|\widehat{f}\|_\infty|\varphi (x)|d\mu (\varphi)\\ 
 &\le \|\widehat{f}\|_\infty \displaystyle\int_{\mathcal{S}^+}|m(\varphi)||\varphi (x)|d\mu (\varphi)
\end{align*}
Since  $\|\widehat{f}\|_\infty \le \|f\|_{{L^{1,\natural}(G)}}$ (Theorem \ref{Riemann-Lebesgue}) and $\forall x\in G,|\varphi (x)| \le 1$, then $$|T_mf(x)|\le \|f\|_{{L^{1,\natural}(G)}}\|\|m\|_{L^{1}(\mathcal{S}^+)}.$$
 This implies $$\|T_mf\|_{L^{\infty,\natural}(G)}\le  \|f\|_{L^{1,\natural}(G)}\|m\|_{L^{1}(\mathcal{S}^+)}.$$
Thus,  $T_m : L^{1,\natural}(G)\longrightarrow L^{\infty,\natural}(G)$ is bounded and  $$\|T_m\|\le \|m\|_{L^{1}(\mathcal{S}^+)}$$
 \end{proof}

\begin{theorem}\label{theorem:boundedLinfty}
If  $m\in L^{\infty}(\mathcal{S}^+)$,  then $T_m : L^{2,\natural}(G)\longrightarrow L^{2,\natural}(G)$ is bounded and $$\|T_m\|\le \|m\|_{L^\infty(\mathcal{S}^+)}.$$
\end{theorem}
\begin{proof}
Thanks to the Plancherel theorem (Theorem \ref{Plancherel}),we have 
\begin{align*}
\|T_mf\|^2_{L^{2,\natural}(G)}&= \|\mathcal{F}(T_mf)\|^2_{L^{2}(\mathcal{S}^+)}\\ 
&= \|
m\widehat{f}\|^2_{L^{2}(\mathcal{S}^+)}\\
 &= \displaystyle\int_{\mathcal{S}^+}|m(\varphi)|^2|\widehat{f}(\varphi)|^2d\mu (\varphi)\\ 
 &\le \|m\|^2_{L^\infty(\mathcal{S}^+)}  \|\widehat{f}\|^2_{L^2(\mathcal{S}^+)} \\
 &=\|m\|^2_{L^\infty(\mathcal{S}^+)}  \|f\|^2_{L^{2,\natural}(G)}.
\end{align*}
Thus, $T_m : L^{2,\natural}(G)\longrightarrow L^{2,\natural}(G)$ is bounded and $$\|T_m\|\le \|m\|_{L^\infty(\mathcal{S}^+)}.$$
\end{proof}

\begin{theorem}\label{theorem:boundedL2} 
If  $m\in L^{1}(\mathcal{S}^+)\cap L^{2}(\mathcal{S}^+)$ is such that $\mathcal{F}^{-1}(m)\in L^{1,\natural}(G)$,  then $T_m : L^{2,\natural}(G)\longrightarrow L^{2,\natural}(G)$ is bounded and $$\|T_m\|\le \|\mathcal{F}^{-1}(m)\|_{L^{1,\natural}(G)}.$$
\end{theorem}

\begin{proof}
For  $f,g\in L^{2,\natural}(G)$, we know from Corollary \ref{Plancherelscalar} that 
$$\langle T_mf,g\rangle_{L^{2,\natural}(G)}= \langle \mathcal{F}(T_mf),  \widehat{g}\rangle_{L^{2}(\mathcal{S}^+)}= \langle m\widehat{f},  \widehat{g}\rangle_{L^{2}(\mathcal{S}^+)}.$$
Now, using  the Cauchy-Schwarz inequality and  some properties  of the convolution product (Subsection \ref{convolution}), we have
\begin{align*}
|\langle T_mf,g\rangle_{L^{2,\natural}(G)}|& =|\langle m\widehat{f},  \widehat{g}\rangle_{L^{2}(\mathcal{S}^+)}|\\
&\le \| m\widehat{f}\|_{L^{2}(\mathcal{S}^+)} \|\widehat{g}\|_{L^{2}(\mathcal{S}^+)}\\
&= \| m\widehat{f}\|_{L^{2}(\mathcal{S}^+)} \|g\|_{L^{2,\natural}(G)}\\
&\le \| \mathcal{F}^{-1}(m)\ast f\|_{L^{2,\natural}(G)} \|g\|_{L^{2,\natural}(G)}\\
&\le \| \mathcal{F}^{-1}(m)\|_{L^{1,\natural}(G)} \|f\|_{L^{2,\natural}(G)} \|g\|_{L^{2,\natural}(G)}.
\end{align*}
It follows that 
$T_m : L^{2,\natural}(G)\longrightarrow L^{2,\natural}(G)$ is bounded and $$\|T_m\|\le \|\mathcal{F}^{-1}(m)\|_{L^{1,\natural}(G)}.$$
\end{proof}

\begin{remark}{\rm
We proved in 
\begin{enumerate}
\item Theorem \ref{theorem:boundedL1} that if $m\in L^{1}(\mathcal{S}^+)$, then  the operator $T_m : L^{1,\natural}(G)\longrightarrow L^{\infty,\natural}(G)$ is bounded and $$\|T_m\|\le \|m\|_{L^{1}(\mathcal{S}^+)}.$$
\item Theorem \ref{theorem:boundedL2} that if $m\in L^{1}(\mathcal{S}^+)\cap L^{2}(\mathcal{S}^+)$ is such that $\mathcal{F}^{-1}(m)\in L^{1,\natural}(G)$, then the operator  $T_m : L^{2,\natural}(G)\longrightarrow L^{2,\natural}(G)$ is bounded and $$\|T_m\|\le \|\mathcal{F}^{-1}(m)\|_{L^{1,\natural}(G)}.$$
\end{enumerate}
}\end{remark}

We  extended the result for $1<p<2$.
\begin{theorem}\label{theorem:interpolationmdansL1}
Let  $m\in L^{1}(\mathcal{S}^+)\cap L^{2}(\mathcal{S}^+)$ such that $\mathcal{F}^{-1}(m)\in L^{1,\natural}(G).$ 
If $1<p<2$, then $T_m : L^{p,\natural}(G)\longrightarrow L^{q,\natural}(G)$ is bounded and $$\|T_m\|\le \|m\|_{L^{1}(\mathcal{S}^+)}^{\frac{2-p}{p}}\|\mathcal{F}^{-1}(m)\|_{L^{1,\natural}(G)}^{\frac{2p-2}{p}}$$
where $q$ is such that  $\displaystyle\frac{1}{p}+\frac{1}{q}=1$.  
\end{theorem}

\begin{proof}
We apply the Riesz-Thorin interpolation theorem  with $p_0=1, p_1=2, q_0=\infty, q_1=2$. For $0<\theta<1$, we have  
$\displaystyle\frac{1}{p} = \frac{2-\theta}{2}$ and $\displaystyle\frac{1}{q} = \frac{\theta}{2}$. Therefore,  $1<p<2$ and $\displaystyle\frac{1}{p}+\frac{1}{q}=1$. The result follows.  
\end{proof}

\section{The spherical Fourier multipliers and the Schatten-von Neumann classes}\label{The spherical Fourier multiplier operators and the Schatten-von Neumann classes}

In this section, we assume that  $G$ is a compact group with a normalized Haar measure. We still assume that there is a compact subgroup of $G$ such that $(G,K)$ is a Gelfand pair.  Since $G$ is compact, the space $\mathcal{S}^+$ is discret. For this reason, we replace the notation $L^1(\mathcal{S}^+)$ by its discrete version $\ell^1(\mathcal{S}^+)$ and integration on  $\mathcal{S}^+$ is replace by discrete summation.

 The bounded spherical functions on the compact $G$ are square integrable. Therefore, if $f\in L^{2,\natural}(G)$, then the Fourier transform of $f$ is written 
 $$\widehat{f}(\varphi)=\displaystyle\int_G f(x)\varphi(x^{-1})dx=\displaystyle\int_G f(x)\overline{\varphi(x)}dx=\langle f,\varphi\rangle_{L^{2,\natural}(G)}.$$
  Moreover, the expression of the spherical Fourier multiplier  $T_m$  becomes 
 \begin{equation}
T_mf (x)=\sum_{\varphi\in \mathcal{S}^+}m(\varphi)\widehat{f}(\varphi)\varphi (x).
\end{equation}
 
 \begin{theorem}\label{theorem:fourierdansLp}
 If $f\in L^{2,\natural}(G)$, then $\forall q\ge 2,\, \widehat{f}\in \ell^q (\mathcal{S}^+)$ with $\|\widehat{f}\|_{\ell^q(\mathcal{S}^+)}\le \|f\|_{L^{2,\natural}(G)}$.
\end{theorem}
\begin{proof}
Since $G$ is compact, $L^{2,\natural}(G)\subset L^{1,\natural}(G)$ with $\|f\|_{L^{1,\natural}(G)}\le \|f\|_{L^{2,\natural}(G)}$ for all  $f\in  L^{1,\natural}(G)$.
It is well-known that $\widehat{f}$ is a sequence  (indexed by $\mathcal{S}^+$)  which tends  to zero at infinity. Therefore, $\widehat{f}$  is bounded,  that is $\widehat{f}\in \ell^{\infty}(\mathcal{S}^+)$. Moreover, $\|\widehat{f}\|_{\ell^\infty(\mathcal{S}^+)}\le \|f\|_{L^{1,\natural}(G)}$. Therefore,  $\|\widehat{f}\|_{\ell^\infty(\mathcal{S}^+)}\le \|f\|_{L^{2,\natural}(G)}$. So,  the Fourier transform $\mathcal{F}$ is a bounded linear operator from $L^{2,\natural}(G)$ into $\ell^{\infty}(\mathcal{S}^+)$.  Moreover, the Fourier transform $\mathcal{F}$ is a linear isometry from $L^{2,\natural}(G)$ onto $\ell^2(\mathcal{S}^+)$. Applying the Riesz-Thorin interpolation theorem (Theorem \ref{Riesz-Thorin}), we obtain that $\forall q\geq 2, \, \widehat{f}\in \ell^q(\mathcal{S}^+)$ with $\|\widehat{f}\|_{\ell^q(\mathcal{S}^+)}\le \|f\|_{L^{2,\natural}(G)}$.
\end{proof}

\begin{theorem}\label{theorem:mDansLpetLmborne}
If  $m\in \ell^{p}(\mathcal{S}^+),\,1\le p\le  \infty$,  then $T_m : L^{2,\natural}(G)\longrightarrow L^{2,\natural}(G)$ is bounded and $$\|T_m\|\le \|m\|_{\ell^p(\mathcal{S}^+)}.$$
\end{theorem}
\begin{proof}
Let  $f,g\in L^{2,\natural}(G)$.  We have 
$$\langle T_mf,g\rangle_{L^{2,\natural}(G)}= \langle \mathcal{F}(T_mf),  \widehat{g}\rangle_{\ell^2(\mathcal{S}^+)}= \langle m\widehat{f},  \widehat{g}\rangle_{\ell^2(\mathcal{S}^+)}.$$
Therefore,
\begin{align*}
|\langle T_mf,g\rangle_{L^{2,\natural}(G)}| 
&= |\langle m\widehat{f},  \widehat{g}\rangle_{\ell^2(\mathcal{S}^+)}|\\
&= \left|\sum_{\varphi\in \mathcal{S}^+}m(\varphi)\widehat{f}(\varphi)\overline{\widehat{g}(\varphi)}\right|\\
&\le \sum_{\varphi\in \mathcal{S}^+}|m(\varphi)||\widehat{f}(\varphi)||\widehat{g}(\varphi)|.
\end{align*}
Now, applying the generalized H\"older inequality (discrete version) with $q=\displaystyle\frac{2p}{p-1}$ so that $q\ge 2$ and $\displaystyle\frac{1}{p}+\frac{1}{q}+\frac{1}{q}=1$, we obtain
\begin{align*}
|\langle T_mf,g\rangle_{L^{2,\natural}(G)}| &\le \|m\|_{\ell^p(\mathcal{S}^+)}\| \widehat{f}\|_{\ell^q(\mathcal{S}^+)} \|\widehat{g}\|_{\ell^q(\mathcal{S}^+)}\\
&\le\|m\|_{\ell^p(\mathcal{S}^+)} \|f\|_{L^{2,\natural}(G)} \|g\|_{L^{2,\natural}(G)}\,\mbox{(Theorem \ref{theorem:fourierdansLp}}).
\end{align*}
It follows that $T_m : L^{2,\natural}(G)\longrightarrow L^{2,\natural}(G)$ is bounded and $$\|T_m\|\le \|m\|_{\ell^p(\mathcal{S}^+)}.$$
\end{proof}

\begin{theorem}\label{TmTrace-class}
If $m\in \ell^1(S^+)$, then $T_m : L^{2,\natural}(G)\longrightarrow L^{2,\natural}(G)$ is in the  trace class $S_1(L^{2,\natural}(G))$. Moreover, 
 $\|T_m\|_{S_1(L^{2,\natural}(G))}\le 4\|m\|_{\ell^1(\mathcal{S}^+)}$.
\end{theorem}

\begin{proof}

\begin{itemize}
\item Let us assume first that $m$ takes nonnegative values. Let $f\in L^{2,\natural}(G)$. Then, 
$$\langle T_mf,f \rangle_{L^{2,\natural}(G)}=\langle m\widehat{f},\widehat{f} \rangle_{L^{2}(\mathcal{S}^+)}=\sum_{\varphi\in \mathcal{S}^+}m(\varphi)\widehat{f}(\varphi)\overline{\widehat{f}(\varphi)}=\sum_{\varphi\in \mathcal{S}^+}m(\varphi)|\widehat{f}(\varphi)|^2\ge 0.$$
Thus, $T_m$ is a positive operator. 
Let $(e_n)_{n\ge 1}$ be an orthonormal basis of $L^{2,\natural}(G)$. Then, 
\begin{align*}
\langle T_me_n,e_n\rangle_{L^{2,\natural}(G)}
&= \displaystyle\int_{G}T_me_n(x)\overline{e_n(x)}dx \\
&= \displaystyle\int_{G}\sum_{\varphi\in \mathcal{S}^+}m(\varphi)\widehat{e_n}(\varphi)\varphi(x)\overline{e_n(x)}dx.
\end{align*}
We want to apply the Dominated convergence Theorem. 

On one hand, we have
\begin{align*}
|m(\varphi)\widehat{e_n}(\varphi)\varphi(x)|&=|m(\varphi)||\langle e_n,\varphi\rangle||\varphi(x)|\\
&\le |m(\varphi)|\|e_n\|_{L^{2,\natural}(G)} \|\varphi\|_{L^{2,\natural}(G)} |\varphi(x)|\, (\mbox{Cauchy-Schwarz})\\
&\le m(\varphi). 
\end{align*}
because $\|e_n\|_{L^{2,\natural}(G)}=1$,  $\forall x\in G,  |\varphi (x)|\le 1$ and $\|\varphi\|_{L^{2,\natural}(G)}^2=\displaystyle\int_{G}|\varphi(x)|^2dx\le 1$. Since $m\in \ell^1(\mathcal{S}^+)$, then the series $\sum\limits_{\varphi\in \mathcal{S}^+}m(\varphi)$ converges.

On the other hand,
let $F$ be a finite subset of $\mathcal{S}^+$. We have 
\begin{align*}
\left|\sum_{\varphi\in F}m(\varphi)\widehat{e}(\varphi)\varphi (x)\right|&\le \sum_{\varphi\in F}|m(\varphi)||\widehat{e_n}(\varphi)||\varphi (x)|\\
&=\sum_{\varphi\in F}|m(\varphi)||\langle e_n, \varphi\rangle||\varphi (x)|\\
&\le \sum_{\varphi\in F}|m(\varphi)| \|e_n\|_{L^{2,\natural}(G)} \|\varphi\|_{L^{2,\natural}(G)}|\varphi (x)|\\
&\le \|m\|_{\ell^1(\mathcal{S}^+)}.
\end{align*}
Since $G$ is compact, the constant function $x\longmapsto \|m\|_{\ell^1(\mathcal{S}^+)}$ defined on $G$ is integrable. Now, we apply the Dominated Convergence Theorem to obtain 
\begin{align*}
\langle T_me_n,e_n\rangle &=\sum_{\varphi\in \mathcal{S}^+}\displaystyle\int_{G}m(\varphi)\widehat{e_n}(\varphi)\varphi(x)\overline{e_n(x)}dx\\
&=\sum_{\varphi\in \mathcal{S}^+}m(\varphi)\widehat{e_n}(\varphi)\displaystyle\int_{G} \varphi(x)\overline{e_n(x)}dx \\
&= \sum_{\varphi\in \mathcal{S}^+}m(\varphi)\widehat{e_n}(\varphi)\langle\varphi, e_n\rangle \\
&= \sum_{\varphi\in \mathcal{S}^+}m(\varphi)\langle e_n,\varphi\rangle\langle\varphi, e_n\rangle\\
&= \sum_{\varphi\in \mathcal{S}^+}m(\varphi)|\langle\varphi, e_n\rangle|^2.
\end{align*}
It follows that
\begin{align*}
\sum_{n=1}^{\infty}|\langle T_me_n,e_n\rangle|
&\le \sum_{n=1}^{\infty}\sum_{\varphi\in \mathcal{S}^+}|m(\varphi)||\langle\varphi, e_n\rangle|^2  \\
&=\sum_{\varphi\in \mathcal{S}^+}|m(\varphi)|\sum_{n=1}^{\infty}|\langle\varphi, e_n\rangle|^2  \\
&= \sum_{\varphi\in \mathcal{S}^+}|m(\varphi)|\|\varphi\|^2_{L^{2,\natural}(G)}\,(\mbox{Parseval identity})\\
&\le \|m\|_{\ell^1(\mathcal{S}^+)}.
\end{align*}
Thus, the operator $T_m$ is in the trace class $S_1(L^{2,\natural}(G))$ and its trace class norm satisfies $\|T_m\|_{S_1(L^{2,\natural}(G))}\le \|m\|_{\ell^1(\mathcal{S}^+)}.$
\item Assume that $m$ takes  real values.  Set 
$$m_+(\varphi)=\max\{m(\varphi),0\} \mbox{ and  }m_-(\varphi)=-\min\{m(\varphi),0\}, \, \varphi\in \mathcal{S}^+.$$
Then, $T_m=T_{m_+}-T_{m_-}$. Thus, $T_m$ is  in the trace class $S_1(L^{2,\natural}(G))$ since the latter is a vector space. Moreover, 
\begin{align*}
\|T_m\|_{S_1(L^{2,\natural}(G))}&=\|T_{m_+}-T_{m_-}\|_{S_1(L^{2,\natural}(G))}\\
&\le  \|T_{m_+}\|_{S_1(L^{2,\natural}(G))}+\|T_{m_-}\|_{S_1(L^{2,\natural}(G))}\\
&\le \|m_+\|_{\ell^1(\mathcal{S}^+)} + \|m_-\|_{\ell^1(\mathcal{S}^+)}\\
&\le 2\|m\|_{\ell^1(\mathcal{S}^+)}.
\end{align*}
\item Assume that $m$ takes  complex values. Then,  $m=m_1+im_2$  where $m_1$ and $m_2$ are real-valued sequences indexed by $\mathcal{S}^+$.  Then, 
\begin{align*}
\|T_m\|_{S_1(L^{2,\natural}(G))}&=\|T_{m_1}+iT_{m_2}\|_{S_1(L^{2,\natural}(G))}\\
&\le \|T_{m_1}\|_{S_1(L^{2,\natural}(G))}+  \|T_{m_2}\|_{S_1(L^{2,\natural}(G))}\\
&\le 2\|m_1\|_{\ell^1(\mathcal{S}^+)}  + 2\|m_2\|_{\ell^1(\mathcal{S}^+)}\\
&\le 4\|m\|_{\ell^1(\mathcal{S}^+)}.
\end{align*}
\end{itemize}
\end{proof}

\begin{theorem}\label{LmCompact}
If $m\in \ell^p(\mathcal{S}^+),\, 1\le p\le \infty$, then $T_m : L^{2,\natural}(G)\longrightarrow L^{2,\natural}(G)$ is a compact operator.
\end{theorem}

\begin{proof}
Let $m\in \ell^p(\mathcal{S}^+),\, 1\le p\le \infty$. From Theorem \ref{theorem:mDansLpetLmborne},
 $\|T_m\|\le\|m\|_{\ell^p(\mathcal{S}^+)}.$
 Let $\mathcal{D}$ be the set of  sequences (indexed by $\mathcal{S}^+$) of numbers  which are zero from a certain rank. The set $\mathcal{D}$  is a dense subset of $\ell^p(\mathcal{S}^+)$. Therefore, there exists a sequence $(m_k)\subset \mathcal{D}$ such that $m_k$ tends to $m$ in $\ell^p(\mathcal{S}^+)$ when $k$ goes to $\infty$.
 We have 
 $$\|T_{m_k}-T_m\|\le \|m_k-m\|_{\ell^p(\mathcal{S}^+)}.$$
 Then, $T_{m_k}$ tends to  $T_m$ in $\mathcal{B}(L^{2,\natural}(G))$  as $k$ goes to $\infty$, where $\mathcal{B}(L^{2,\natural}(G))$ is the space of bounded operators on the Hilbert space $L^{2,\natural}(G)$.
 However,  $m_k$ is in $\ell^1(\mathcal{S}^+)$.  
Therefore, from Theorem \ref{TmTrace-class}, $T_{m_k}$ is in the trace class $S_1(L^{2,\natural}(G))$.  Thus,  $T_{m_k}$ is compact.  Therefore, $T_m$ is a compact operator since it is  a limit in $\mathcal{B}(L^{2,\natural}(G))$ of compact operators.
\end{proof}

\begin{theorem}\label{TmDansSp}
If  $m\in \ell^p(\mathcal{S}^+),\,1\le p \le\infty$,  then $T_m : L^{2,\natural}(G)\longrightarrow L^{2,\natural}(G)$ is in the $p$-Schatten-von Neumann class $S_p(L^{2,\natural}(G))$ and 
$$\|T_m\|_{S_p(L^{2,\natural}(G))}\le 4^{\frac{1}{p}}\|m\|_{\ell^p(\mathcal{S}^+)}.$$
\end{theorem}
\begin{proof} From Theorem \ref{TmTrace-class},  we have 
$$\|T_m\|_{S_1(L^{2,\natural}(G))}\le 4\|m\|_{\ell^1(\mathcal{S}^+)},\, m\in \ell^1(\mathcal{S}^+).$$
and from Theorem \ref{theorem:boundedLinfty}, using the fact that $S_\infty(L^{2,\natural}(G))=\mathcal{B}(L^{2,\natural}(G))$,   we have 
 $$\|T_m\|_{S_\infty(L^{2,\natural}(G))}\le \|m\|_{\ell^\infty(\mathcal{S}^+)},\, m\in \ell^\infty(\mathcal{S}^+).$$
Then, by interpolation (Theorem \ref{Riesz-Thorin} and Theorem \ref{InterpolationParticular})  we obtain that 
$T_m$ is in $S_p(L^{2,\natural}(G))$ for $1\le p \le\infty$ and 
$$\|T_m\|_{S_p(L^{2,\natural}(G))}\le 4^{\frac{1}{p}}\|m\|_{\ell^p(\mathcal{S}^+)}.$$
 
\end{proof}

\end{document}